\documentclass[12pt, reqno]{amsart}
\usepackage{amsmath, amsthm, amscd, amsfonts, amssymb, graphicx, color}
\usepackage[bookmarksnumbered, colorlinks, plainpages]{hyperref}
\hypersetup{colorlinks=true,linkcolor=red, anchorcolor=green, citecolor=cyan, urlcolor=red, filecolor=magenta, pdftoolbar=true}

\newtheorem{theorem}{Theorem}[section]
\newtheorem{lemma}[theorem]{Lemma}
\newtheorem{proposition}[theorem]{Proposition}

\theoremstyle{definition}
\newtheorem{definition}[theorem]{Definition}

\theoremstyle{remark}
\newtheorem{remark}[theorem]{Remark}
\numberwithin{equation}{section}

\newcommand{\real}{{\mathbb R}}

\newcommand{\A}{{\mathcal A}}

\newcommand{\E}{{\mathcal E}}

\newcommand{\M}{{\mathcal M}}
\newcommand{\N}{{\mathcal N}}
\newcommand{\R}{{\mathcal R}}

\newcommand{\tr}{\mbox{\rm tr}}

\newcommand{\8}{\infty}
\newcommand{\el}{\ell}

\newcommand{\be}{\begin{eqnarray*}}
\newcommand{\ee}{\end{eqnarray*}}
\newcommand{\beq}{\begin{equation}}
\newcommand{\eeq}{\end{equation}}
\newcommand{\beqn}{\begin{equation*}}
\newcommand{\eeqn}{\end{equation*}}
\newcommand{\bsp}{\begin{split}}
\newcommand{\esp}{\end{split}}

\begin{document}

\setcounter{page}{1}

\title[Interpolation of Haagerup noncommutative Hardy spaces]{Interpolation of Haagerup noncommutative Hardy spaces}

\author[Turdebek N. Bekjan \MakeLowercase{and}  Madi Raikhan]{Turdebek N. Bekjan$^1$$^{*}$ \MakeLowercase{and} Madi Raikhan$^2$}

\address{$^{1}$College of Mathematics and Systems Science, Xinjiang
University, Urumqi 830046, China.}
\email{\textcolor[rgb]{0.00,0.00,0.84}{ bekjant@yahoo.com;
bek@xju.edu.cn}}

\address{$^{2}$ Faculty of Mechanics and Mathematics, L.N. Gumilyov Eurasian National University,  Astana 010008, Kazakhstan.}
\email{\textcolor[rgb]{0.00,0.00,0.84}{mady1029@yandex.kz}}


\let\thefootnote\relax\footnote{Copyright 2016 by the Tusi Mathematical Research Group.}



\date{Received: xxxxxx; Revised: yyyyyy; Accepted: zzzzzz.
\newline \indent $^{*}$Corresponding author}

\begin{abstract}
Let $\mathcal{M}$ be a $\sigma$-finite von Neumann algebra, equipped with a
 normal  faithful state $\varphi$,  and  let $\mathcal{A}$ be maximal subdiagonal algebra of $\mathcal{M}$.
 We prove  Stein-Weiss type
interpolation theorem of   Haagerup
noncommutative $H^{p}$-spaces associated with $\A$.
\end{abstract} \maketitle

\section{Introduction}
 In \cite{A}, for von Neumann algebra $\mathcal{M}$  with a faithful, normal
finite trace, Arveson introduced the notion of finite, maximal, subdiagonal
algebras  $\mathcal{A}$ of $\mathcal{M},$ as  noncommutative analogues of
weak*-Dirichlet algebras. In \cite{MW1}, Marsalli and West defined the noncommutative $H^p$-space
by the closure of $\A$ in the noncommutative $L^p$-space $L^p(\M)$ and extended  some results of the classical Hardy spaces on the torus
to this noncommutative setting. Using some of Arveson's ideas, Labuschagne \cite{L1}
showed that in the context of finite von Neumann algebras, these maximal subdiagonal
algebras satisfy a Szeg$\ddot{o}$ formula. Pursuant to this breakthrough, the theory of the (noncommutative) $H^{p}$-spaces associated with such algebras   has been rapidly developed. Many classical results of $H^p$-spaces associated with weak* Dirichlet algebras have been transferred to the noncommutative setting (\cite{BL1}).

In \cite{J1, J2}, Ji studied  Haagerup noncommutative
$H^{p}$-spaces based on Haagerup's noncommutative
$L^p$-space $L^p(\M)$ associated with a $\sigma$-finite von Neumann algebra $\M$ and a maximal subdiagonal
algebra $\A$ of $\M$.  Ji extended some results in
\cite{MW1} to the Haagerup noncommutative
$H^{p}$-space case (see also \cite{JS}). Labuschagne \cite{L2} showed that a Beurling type theory of invariant subspaces of
Haagerup noncommutative $H^2$ spaces holds. The first-named author \cite{B1} proved  a Szeg\"o type factorization theorem for   Haagerup
noncommutative $H^{p}$-spaces.

 Kosaki \cite{Kos} proved  a Haagerup noncommutative
$L^{p}$-spaces analogue of the classical Stein-Weiss interpolation theorem. In general, the real interpolation theorem of classical $L^p$-spaces is no longer valid for the Haagerup noncommutative
$L^{p}$-spaces (see Example 3.3 in\cite{PX}).

 In \cite{PX}, Pisier and Xu obtained
noncommutative version of P. Jones' theorem  for noncommutative Hardy spaces associated with a finite
subdiagonal algebra. It is stated in \cite{PX}
without proof (see the remark following Lemma 8.5 there). In \cite{B}, the first-named author extends the results of the real interpolation
method to the semifinite case. Also, the results for the complex interpolation
method are extended to the semifinite case (see \cite{BO}). The main objective of this paper is to prove Stein-Weiss type
interpolation theorem of   Haagerup
noncommutative $H^{p}$-spaces.

The organization of the paper is as follows. In Section 2, we
give some definitions and  related results of  Haagerup
noncommutative $H^{p}$-spaces.   The Stein-Weiss type
interpolation theorems of  Haagerup
noncommutative $H^{p}$-space are presented in Section 3.

\section{Preliminaries}

We use standard notation in operator algebras. We refer to \cite{PT,T3}
 for modular theory, to \cite{H1} for the Haagerup noncommutative
$L^p$-spaces and to \cite{J1,J2} for the Haagerup noncommutative
$H^p$-spaces. Let us recall some basic facts about these spaces and fix the relevant
notation used throughout this paper.
 Let $\mathcal{M}$ be a $\sigma$-finite von Neumann algebra on a complex
Hilbert space $\mathcal{H}$, equipped with a distinguished
 normal  faithful state $\varphi$. Recall briefly the definition of Haagerup  noncommutative $L^p$ spaces  associated with $\M$.
Let $\{\sigma_{t}^\varphi\}_{t\in\mathbb{R}}$
 be the one parameter modular automorphism
group of $\mathcal{M}$ associated with $\varphi$,  and let $\mathcal{N}$ denote
the crossed product$
\mathcal{M}\rtimes_{\sigma^\varphi}\mathbb{R}$
of $\mathcal{M}$ by
$\{\sigma_{t}^\varphi\}_{t\in\mathbb{R}}.$ Then
$\mathcal{N}$ is the von Newmann algebra acting on the Hilbert space
$L^{2}(\mathbb{R},\mathcal{H}),$ generated by
$$
\left\{ \pi(x):\;x\in\mathcal{M}\right\}\cup\left\{\lambda(s):\;s\in \mathbb{R}\right\},
$$
where  the operator $\pi(x)$ is defined by
$$
(\pi(x)\xi)(t)=\sigma_{-t}^\varphi(x)\xi(t),\quad \forall\xi\in
L^{2}(\mathbb{R},\mathcal{H}),\;\forall t\in\mathbb{R},
$$
and the operator $\lambda(s)$ is defined by

$$
(\lambda(s)\xi)(t)=\xi(t-s),\quad \forall\xi\in
L^{2}(\mathbb{R},\mathcal{H}),\;\forall t\in\mathbb{R}.
$$
We identify $\mathcal{M}$ with its image $\pi(\mathcal{M})$ in $\mathcal{N}$. The operators $\pi(x)$ and $\lambda(t)$ satisfy
the following commutation relation:
\be
\lambda(t)\pi(x)\lambda(t)^*=\pi(\sigma_{t}^\varphi(x)), \quad\forall t\in \real, \;\forall x\in\M.
\ee
Then the modular automorphism group $\{\sigma_{t}^\varphi\}_{t\in\mathbb{R}}$ is given
by
$$
\sigma_{t}^\varphi(x)=\lambda(t)x\lambda^{\ast}(t), \quad x\in\mathcal{M},\; t\in\mathbb{R}.
$$

We denote by $\{\hat{\sigma}_{t}\}_{t\in\mathbb{R}}$ the dual action of $\mathbb{R}$
on $\mathcal{N}$.  Then the dual action $\hat{\sigma}_{t}$ is uniquely determined by the
following conditions: for any $x\in\M$ and $s\in\mathbb{R}$,
$$
\hat{\sigma}_{t}(x)=x \quad\mbox{and}\quad
\hat{\sigma}_{t}(\lambda(s))=e^{-ist}\lambda(s),\quad \forall t\in\mathbb{R}.
$$
Hence
$$
\mathcal{M}=\{x\in\mathcal{N}: \; \hat{\sigma}_{t}(x)=x,\;\forall
t\in\mathbb{R}\}.
$$
Since $\mathcal{N}$ is semi-finite (cf. \cite{PT}), there exists a
 normal semi-finite faithful trace $\tau$ on $\mathcal{N}$ satisfying the
equation
$$
\tau\circ\hat{\sigma}_{t}=e^{-t}\tau,\quad \forall t\in\mathbb{R}.
$$

Also recall that any  normal semi-finite faithful weight $\psi$ on
$\mathcal{M}$ induces a dual normal semi-finite weight $\hat{\psi}$ on
$\mathcal{N}.$ Then $\hat{\psi}$ admits a Radon-Nikodym derivative with
respect to $\tau$ (cf. \cite{PT}). In particular the dual weight
$\hat{\varphi}$ of our distinguished state $\varphi$ has the
Radon-Nikodym derivative $D$ with respect to $\tau.$ Then
$$
\hat{\varphi}(x)=\tau(Dx),\quad x\in \mathcal{N}_{+}.
$$
Recall that $D$ is an invertible positive selfadjoint operator on
$L^{2}(\mathbb{R},\mathcal{H}),$ affiliated with $\mathcal{N},$ and that the
regular representation $\lambda(t)$ above is given by
$$
\lambda(t)=D^{it},\quad \forall t\in \mathbb{R}.
$$
 Now we define
Haagerup noncommutative $L^{p}$-spaces. Recall that $L^{0}(\mathcal{N},\tau)$
denotes the topological $\ast$-algebra of all operators on
$L^{2}(\mathbb{R},\mathcal{H})$ measurable with respect to
$(\mathcal{N},\tau).$ Then the Haagerup noncommutative $L^{p}$-spaces,
$0<p\leq\infty$, are defined by
$$
L^{p}(\mathcal{M},\varphi)=\{x\in L^{0}(\mathcal{N},\tau):
\;\hat{\sigma}_{t}(x)=e^{-\frac{t}{p}}x,\;\forall t\in\mathbb{R}\}.
$$
It is clear that  $L^{p}(\mathcal{M},\varphi)$ is a vector subspace of
$L^{0}(\mathcal{N},\tau),$ invariant under the $\ast$-operation. The algebraic
structure of   $L^{p}(\mathcal{M},\varphi)$ is inherited from that of
$L^{0}(\mathcal{N},\tau).$  Let $x\in L^{p}(\mathcal{M},\varphi)$ and $x=u|x|$
be its polar decomposition, where
$$
|x|=(x^{\ast}x)^{\frac{1}{2}}
$$
is the
modulus of $x$. Then
 $$
 u\in\mathcal{M}\quad\mbox{and}\quad|x|\in L^{p}(\mathcal{M},\varphi).
 $$
Recall that
$$
L^{\infty}(\mathcal{M},\varphi)=\mathcal{M}.
$$
As mentioned previously, for any $\psi\in\mathcal{M}_{\ast}^{+},$ the dual
weight $\hat{\psi}$ has a Radon-Nikodym derivative with respect to
$\tau,$ denoted by $h_{\psi}:$
$$
\hat{\psi}(x)=\tau(h_{\psi}x),\quad x\in\mathcal{N}_{+}.
$$
Then
$$
h_{\psi}\in L^{0}(\mathcal{N},\tau)
$$
and
$$
\hat{\sigma}_{t}(h_{\psi})=e^{-t}h_{\psi},\quad\forall
t\in\mathbb{R}.
$$
So
$$
h_{\psi}\in L^{1}(\mathcal{M},\varphi)_{+}.
$$
Hence, this correspondence
between $\mathcal{M}_{\ast}^{+}$ and $ L^{1}(\mathcal{M},\varphi)_{+}$ extends
to a bijection between $\mathcal{M}_{\ast}$ and $ L^{1}(\mathcal{M},\varphi).$
So that for $\psi\in\mathcal{M}_{\ast}, $ if  $\psi=u|\psi|$ is its polar
decomposition, the corresponding  $h_{\psi} \in L^{1}(\mathcal{N},\tau)$ admits
the polar decomposition
$$
h_{\psi}=u|h_{\psi}|=uh_{|\psi|}.
$$
Then the norm on
$L^{1}(\mathcal{M},\varphi),$ is defined as
$$
\parallel h_{\psi}\parallel_{1}=|\psi|(1)=\parallel
\psi\parallel_{\ast},\quad \forall\psi\in\mathcal{M}_{\ast}.
$$
In this way,
$$
L^{1}(\mathcal{M},\varphi)=\mathcal{M}_{\ast}\;\mbox{isometrically}.
$$
For
$0<p<\infty,$ we define
$$
\parallel x\parallel_{p}=\parallel |x|^{p}\parallel_{1}^{\frac{1}{p}},\quad \forall x\in
L^{p}(\mathcal{M},\varphi),
$$
since $x\in L^{p}(\mathcal{M},\varphi)$ if and only if $|x|\in
L^{p}(\mathcal{M},\varphi)$.
Then for $1\leq p<\infty$ (resp. $0<p<1$).
$$
(L^{p}(\mathcal{M},\varphi),\;\|\cdot\|_p)
$$
is a Banach space (resp. a quasi-Banach space), and
the norm satisfies the following equations:
$$
\| x\|_{p}=\| x^{\ast}\|_{p}=\| |x|\|_{p},\quad \forall x\in
L^{p}(\mathcal{M},\varphi).
$$
 The usual Holder inequality also holds for the $L^{p}(\mathcal{M},\varphi)$ spaces. Let
 $$
 0<p,q,r\leq\infty,\;\frac{1}{r}=\frac{1}{p}+\frac{1}{q}.
 $$
If $x\in L^{p}(\mathcal{M},\varphi),\; y\in L^{q}(\mathcal{M},\varphi)$, then
$$
xy\in L^{r}(\mathcal{M},\varphi),\quad\mbox{and}\quad\|
xy\|_{r}\leq\| x\|_{p}\| x\|_{q}.
$$

 It is well known that  $L^{p}(\mathcal{M},\varphi)$ is independent of
$\varphi$ up to isometry. So, following Haagerup, we will use the notation
$L^{p}(\mathcal{M})$ for the abstract Haagerup noncommutative $L^{p}$-space
$L^{p}(\mathcal{M},\varphi)$. To describe duality of Haagerup noncommutative
$L^{p}$-space, we use the distinguished linear function on $L^{1}(\mathcal{M})$
which is defined by
$$
tr(x)=\psi_{x}(1),\quad \forall x\in L^{1}(\mathcal{M}),
$$
where $\psi_{x}\in\mathcal{M}_{\ast}$ is the unique normal functional
associated with $x$ by the above identification between $\mathcal{M}_{\ast}$
and $L^{1}(\mathcal{M}).$ Then $tr$ is a continuous functional on
$L^{1}(\mathcal{M})$ satisfying
$$
|tr(x)|\leq tr(|x|)=\parallel x\parallel_{1},\quad \forall x\in L^{1}(\mathcal{M}).
$$
Let
$$
1\leq p<\infty,\;\frac{1}{p}+\frac{1}{q}=1.
$$
Then $tr$ have the following  property:
$$
tr(xy)=tr(yx),\quad \forall x\in L^{p}(\mathcal{M}),\;\forall y\in
L^{q}(\mathcal{M}).
$$
The bilinear form $(x,y)\mapsto
tr(xy)$ defines a duality bracket between $L^{p}(\mathcal{M})$ and
$L^{q}(\mathcal{M})$, for which
$$
(L^{p}(\mathcal{M}))'=L^{q}(\mathcal{M})\quad  \mbox{isometrically for all}\;1\leq
p<\infty.
$$

Moreover, our distinguished state $\varphi$ can be recovered from $tr$
(recalling that $D$ is the Radon-Nikodym derivative of $\hat{\varphi}$ with
respect to $\tau$),  that is,
$$
\varphi(x)=tr(Dx),\quad \forall x\in \mathcal{M}.
$$

Let $0<p\le\infty$,\; $K \subset L^{p}(\mathcal{M})$.  We denote by $[K]_{p}$  the closed linear span of  $K$ in $L^{p}(\mathcal{M})$ (relative to the w*-topology for $p =\infty$) and set
$$
J(K)=\{x^{\ast}:\;x\in K\}.
$$

For $0<p<\infty,\; 0\le \eta\le1$, we have that
$$
L^{p}(\mathcal{M})=[D^{\frac{1-\eta}{p}}\M D^{\frac{\eta}{p}}]_p.
$$

 Let $\mathcal{D}$ be a von Neumann
 subalgebra of $\mathcal{M}.$ Let $\mathcal{E}$ be the (unique) normal faithful
 conditional expectation of $\mathcal{M}$ with respect to $\mathcal{D}$ which
 leaves $\varphi$ invariant.

\begin{definition} A w*-closed subalgebra $\mathcal{A}$ of $\mathcal{M}$
is called a subdiagonal algebra of $\mathcal{M}$ with respect to
$\mathcal{E}$(or to $\mathcal{D}$)  if \begin{enumerate}

\item $\mathcal{A}+ J(\mathcal{A})$ is w*-dense in  $\mathcal{M}$,

\item $\mathcal{E}(xy)=\mathcal{E}(x)\mathcal{E}(y),\; \forall\;x,y\in
\mathcal{A},$

\item $\mathcal{A}\cap J(\mathcal{A})=\mathcal{D},$
\end{enumerate}
The algebra $\mathcal{D}$ is  called the diagonal of $\mathcal{A}.$
\end{definition}

We say that $\mathcal{A}$ is a maximal subdiagonal algebra in $\mathcal{M}$ with
respect to $\mathcal{E}$ in the case that $\mathcal{A}$ is not properly contained
in any other subalgebra of $\mathcal{M}$ which is a subdiagonal with respect to
$\mathcal{E}.$  Let
$$\mathcal{A}_{0}=\{x\in \mathcal{A}:\;\mathcal{E}(x)=0\}
$$
and
$$
 \mathcal{A}_{m}=\{x\in
\mathcal{M}:\;\mathcal{E}(yxz)=\mathcal{E}(yxz)=0,\;\forall y\in
\mathcal{A},\;\forall z\in \mathcal{A}_{0}\}.
$$
By Theorem  2.2.1 in \cite{A},  $\mathcal{A}_{m}$ is a maximal
subdiagonal algebra of $\mathcal{M}$ with respect to $\mathcal{E}$ containing
$\mathcal{A}.$

 It follows from Theorem 2.4 in \cite{JOS} and Theorem 1.1 in \cite{X} that a subdiagonal algebra $\A$ of $\mathcal{M}$ with respect to
 $\mathcal{D}$ is maximal if and only if
\beq\label{maximal}
\sigma_{t}^\varphi(\A)=\A,\quad \forall t\in\mathbb{R}.
\eeq

In this paper $\mathcal{A}$ always denotes a maximal subdiagonal algebra in $\mathcal{M}$ with
respect to $\mathcal{E}$.

\begin{definition}\label{def:hp}   For $0<p<\infty$,  we define the Haagerup
noncommutative $H^{p}$-space that

$$
H^{p}(\mathcal{A})=[\mathcal{A}D^{\frac{1}{p}}]_{p},\quad H^{p}_{0}(\mathcal{A})=
 [\mathcal{A}_{0}D^{\frac{1}{p}}]_{p}.
$$
\end{definition}

From Proposition 2.1 in \cite{J2}, we know that for $1\le p<\infty,\; 0\le \eta\le1$,
\beq\label{equalityHp}
H^{p}(\mathcal{A})=[D^{\frac{1-\eta}{p}}\A D^{\frac{\eta}{p}}]_p,\quad H^{p}_0(\mathcal{A})=[D^{\frac{1-\eta}{p}}\A_0 D^{\frac{\eta}{p}}]_p.
\eeq

It is known that
$$
L^{p}(\mathcal{D})=[D^{\frac{1-\eta}{p}}\mathcal{D} D^{\frac{\eta}{p}}]_p,\quad  \forall p\in [1,\infty),\; \forall \eta\in [1,0],
$$
 and the conditional expectation $\mathcal{E}$ extends to a contractive
projection  from $L^p(\mathcal{M})$ onto $L^p(\mathcal{D})$. The extension
will be denoted still by $\mathcal{E}$.
Let
$$
1\le r,\;p,\;q\le\8,\;\frac{1}{r}=\frac{1}{p}+\frac{1}{q}.
$$
Then
\be
\E(xy)=\E(x)\E(y),\quad\forall x\in H^p(\A),\;\forall y\in H^q(\A).
\ee
Indeed, let  $a,b\in\A$. By Proposition 2.3 in \cite{JX}, we have that
$$\begin{array}{rl}
    \E(D^\frac{1}{p}abD^\frac{1}{q}) & =D^\frac{1}{p}\E(ab)D^\frac{1}{q} \\
     & =D^\frac{1}{p}\E(a)\E(b)D^\frac{1}{q}\\
     &=\E(D^\frac{1}{p}a)\E(b D^\frac{1}{q}).
  \end{array}
$$
Hence, by \eqref{equalityHp}, we obtain the desired result.


 We will  repeatedly use the following fact:
\begin{equation}\label{eq:conditional expectation}
  tr(\E(x))=tr(x),\qquad x\in L^1(\M).
\end{equation}

 Let  $\M_1$ be a von Neumann
 subalgebra of $\M$ such that $\M_1$  is invariant under $\sigma^\varphi$ and $\E$, i.e.,
 $$
 \sigma_{t}^\varphi(\M_1)\subset\M_1,\; \forall t\in\real\quad \mbox{and}\quad \E(\M_1)\subset\M_1.
 $$
Since $\M_1$ is $\sigma^\varphi$-invariant, it is well known (cf. \cite{T3})) that there is a (unique)
normal conditional expectation $\Psi$ such that
$$
\varphi\circ\Psi=\varphi.
$$
This conditional expectation  $\Psi:\M\rightarrow\M_1$ commutes with $\{\sigma_{t}^\varphi\}_{t\in\real}$ (cf. \cite{C}), i.e.,
\beq\label{commutes-conditional}
\Psi\circ\sigma_{t}^\varphi=\sigma_{t}^\varphi\circ\Psi, \quad \forall t\in\real.
\eeq
Now let $\phi=\varphi|_{\M_1}$ be the restriction of $\varphi$ to $\M_1$. Using the fact that $\M_1$ is $\sigma^\varphi$-invariant, we obtain that the modular
automorphism group associated with $\phi$ is  $\sigma^\varphi|_{\M_1}$, i.e.,
\beq\label{automorphism-equal}
\sigma^\phi_t=\sigma^\varphi_t|_{\M_1},\quad \forall t\in\real.
\eeq
Hence the crossed product
$$
\N_1=\M_1\rtimes_{\sigma^\phi}\mathbb{R}
$$
 is a von Neumann subalgebra
of
$$
\N=\mathcal{M}\rtimes_{\sigma^\varphi}\mathbb{R}.
$$
Let $\upsilon$ be the canonical  normal semi-finite faithful trace  on $\N_1$. Then $\upsilon$ is equal to the
restriction of $\tau$ to $\N_1$ (recalling that $\tau$  is the canonical trace on $\N$). Observe that the
conditional expectation $\Psi$ extends to a normal faithful conditional expectation $\widehat{\Psi}$
from $\N$ onto $\N_1$, satisfying $\tau\circ\widehat{\Psi}=\tau$, that is, $\upsilon\circ\widehat{\Psi}=\tau$. Let $\widehat{\varphi}$ and $\widehat{\phi}$ be the dual
weights of $\varphi$ and $\phi$ respectively. Then $\widehat{\phi}\circ\widehat{\Psi}=\widehat{\varphi}$, and by \cite{C},
 the Radon-Nikodym derivative of $\widehat{\phi}$ with respect to $\upsilon$ is equal to $D$,
the Radon-Nikodym derivative of $\widehat{\varphi}$ with respect to $\tau$.

Let $\M,\;\varphi,\;\sigma_t^{\varphi},\;\E,\;\mathcal{D},\;\A,\;\M_1,\;\Psi $ be fixed as in the above.

\begin{proposition}\label{pro:Hp-subspace}  Let  $\A$ be $\Psi$- invariant, i.e., $\Psi(\A)\subset\A$. Set $\A_1=\A\cap\M_1$.
 Then   $\mathcal{A}_1$
is  a maximal subdiagonal algebra of $\mathcal{M}_1$ with respect to
 $\E|_{\M_1}$ and
$H^p(\A_1)$ coincides isometrically with a subspace of $H^p(\A)$ for every $0 < p < \8$.
\end{proposition}
\begin{proof}

 It is clear that  $\mathcal{A}_1$ is a w*-closed subalgebra of $\mathcal{M}$ and $\E|_{\M_1}:\M_1\rightarrow\mathcal{D}_1$, where $\mathcal{D}_1=\mathcal{D}\cap\M_1$. Since $\E$ is multiplicative on $\A$, $\E|_{\M_1}$ is multiplicative on $\A_1$. On the other hand,
$$
\A_1\cap J(\A_1)=\A\cap J(A)\cap\M_1=\mathcal{D}\cap\M_1=\mathcal{D}_1.
$$
Thus it remains to show the $\sigma$-weak density of $\A_1+ J(\A_1)$ in $\M_1$. Let $x \in\M_1$. Since
$\A+J(\A)$ is $\sigma$-weakly dense in $\M$, there are $a_i, b_i \in \A$ such that
$$
x = \lim_{i}(a_i + b_i^* )\quad\sigma\mbox{-weakly}.
$$
Then by the normality of $\Psi$, we have
$$
x = \Psi(x) = \lim_{i}(\Psi(a_i) + \Psi(b_i)^*)\quad \sigma\mbox{-weakly}.
$$
Since $\A$ is $\Psi$-invariant,  $\Psi(a_i),\;\Psi(b_i) \in\A \cap\M_1=\A_1$. It follows that $\A_1 + J(\A_1)$
 is $\sigma$-weakly dense in $\M_1$. Thus $A_1$ is a subdiagonal algebra with respect to $\E|_{\M_1}$.

From
$$
\A_1=\Psi(\A_1)\subset\Psi(\A)\subset \A \cap\M_1=\A_1
$$
 follows that $\A_1=\Psi(\A)$. Hence, by \eqref{maximal}, \eqref{commutes-conditional} and \eqref{automorphism-equal}, we have that
$$
\sigma^\phi_t(a)=\sigma^\varphi_t(a)=\sigma^\varphi_t(\Psi(a))=\Psi(\sigma^\varphi_t(a))\in\A_1,\quad \forall a\in\A_1,\;\forall t\in\real.
$$
Using Theorem 1.1 in \cite{X} we obtain that  $\mathcal{A}_1$
is  a maximal subdiagonal algebra of $\mathcal{M}_1$ with respect to
 $\E|_{\M_1}$.

Since $\A_1\subset\A$, we obtain that for every $0 < p < \8$,
$$
H^p(\A_1)=[\A_1D^{\frac{1}{p}}]_p
$$
 coincides isometrically with a subspace of
  $$
  H^p(\A)=[\A D^{\frac{1}{p}}]_p.
  $$

\end{proof}

\begin{remark}\label{rk:independent}
For $1\le p<\8$,  $H^p(\A)$ is independent of $\varphi$ (see Theorem 2.5 in \cite{J1}).
\end{remark}

We also need to give a brief description of Haagerup's reduction theorem (cf. \cite{H2, X}). Let
$$
G=\cup_{n\ge1}2^{-n}\mathbb{Z}.
$$
Then it is
a discrete subgroup  of $\mathbb{R}$. We consider the discrete crossed product
\beq\label{defi:R}
\mathcal{R} =\mathcal{M}\rtimes_{\sigma^\varphi} G
\eeq
 with respect
to $\{\sigma_{t}^\varphi\}_{t\in \real}$ as above by replacing $\mathbb{R}$ by $G$ and $L_2(\mathbb{R},H)$ by $\el_2(G,H)$. Recall that $\R$ is a von Neumann algebra on $\el_2(G,H)$
generated by the operators $\pi(x),\; x\in \M$ and $\lambda(t),\; s \in G$, which are defined by
$$
(\pi(x)\xi)(t)=\sigma_{-t}(x)\xi(t),\quad (\lambda(s)\xi)(t)=\xi(t-s),\quad \; \xi\in\el_2(G,H),\quad \forall t \in G.
$$
Then $\pi$ is a normal faithful representation of $\M$ on
$\el_2(G,H)$ and we identify $\pi(\M)$ with $\M$.  The operators $\pi(x)$ and $\lambda(t)$ satisfy
the following commutation relation:
\beq\label{representation1}
\lambda(t)\pi(x)\lambda(t)^*=\pi(\sigma_{t}^\varphi(x)),\quad\forall t\in G,\;\forall x\in \M.
\eeq
Let $\widehat{\varphi}$ be the dual weight of $\varphi$ on $\mathcal{R}$. Then $\widehat{\varphi}$ is again a faithful
normal state on $\mathcal{R}$ whose restriction on $\mathcal{M}$ is $\varphi$. The modular automorphism group of $\widehat{\varphi}$ is uniquely determined
by
\be
\sigma_{t}^{\widehat{\varphi}}(x)=\sigma_{t}^\varphi(x),\quad \sigma_{t}^{\widehat{\varphi}}(\lambda_s(x))=\lambda_s(x),\quad x\in\M, \;s,t\in G.
\ee
Consequently, $\sigma_{t}^{\widehat{\varphi}}|_{\M}=\sigma_{t}^{\varphi}$, and so $\sigma_{t}^{\widehat{\varphi}}(\M)=\M$ for all $t\in\real$.
 We recall that there is a unique normal faithful conditional expectation $\Phi$ from
$\R$ onto $\M$ determined by
\beq\label{conditional-expectation1}
\Phi(\lambda(t)x)=\left\{\begin{array}{rl}
                           x & \mbox{if}\; t=0 \\
                           0 & \mbox{otherwise}
                         \end{array}
\right.,\quad \forall x\in\M,\;\forall t\in G.
\eeq
It satisfies
$$
\widehat{\varphi}\circ\Phi = \widehat{\varphi},\quad\sigma_{t}^{\widehat{\varphi}}\circ\Phi=\Phi\circ\sigma_{t}^{\widehat{\varphi}}, \quad\forall t\in \real.
$$

We denote  by $\R_{\widehat{\varphi}}$ the centralizer of  $\widehat{\varphi}$  in $\R$. Then
$$
\R_{\widehat{\varphi}}=\{x\in\R:\quad \sigma_{t}^{\widehat{\varphi}}(x)=x, \quad \forall t\in \real\}.
$$
For each $n \in \mathbb{N}$ there exists a unique $b_n\in\mathcal{Z}(\R_{\widehat{\varphi}})$ such that
$$0\le b_n\le2\pi\quad \mbox{and}\quad e^{ib_n}=\lambda(2^{-n}),$$
 where $\mathcal{Z}(\R_{\widehat{\varphi}})$ the center of $\R_{\widehat{\varphi}}$. Let $a_n=2^nb_n$ and define  normal faithful positive functionals on $\R$ by
$$
\varphi_n(x)=\widehat{\varphi}(e^{-a_n}x),\quad \forall x\in \R,\quad \forall n\in\mathbb{N}.
$$
Then $\sigma_{t}^{\varphi_n}$ is $2^{-n}$-periodic and
\beq\label{eq:modular}
\sigma_{t}^{\varphi_n}(x)=e^{-ita_n}\sigma_{t}^{\widehat{\varphi}}(x)e^{ita_n},\quad \forall x\in\R,\;\forall t\in \real,\;\forall n\in\mathbb{N}.
\eeq
Let $\R_n=\R_{\varphi_n}$. Then $\R_n$ is a finite von Neumann algebra equipped with the normal faithful
tracial state $\tau_n=\varphi_n|_{\R_n}$.

Define
$$
\Phi_n(x)=2^n\int_0^{2^{-n}}\sigma_{t}^{\varphi_n}(x)dt,\quad \forall x\in\R.
$$
By the $2^{-n}$-periodicity of $\sigma_{t}^{\varphi_n}$, we have that
\beq\label{eq:conditionexpectation-n}
\Phi_n(x)=\int_0^1\sigma_{t}^{\varphi_n}(x)dt,\quad \forall x\in\R.
\eeq
 Haagerup's reduction theorem (Theorem 2.1 and Lemma 2.7 in \cite{H2}) asserts that $\{\mathcal{R}_n\}_{n\ge1}$ is an
increasing sequence of von Neumann subalgebras of $\mathcal{R}$ with the following properties:
\begin{enumerate}
\item each $\mathcal{R}_n$ is finite;
\item   $\Phi_n$ is a faithful normal conditional expectation from $\mathcal{R}$ onto $\mathcal{R}_n$ such that
$$
\hat{\varphi}\circ\Phi_n = \hat{\varphi},\quad  \sigma_{t}^{\widehat{\varphi}}\circ\Phi_n=\Phi_n\circ\sigma_{t}^{\widehat{\varphi}},\quad \Phi_n \circ\Phi_{n+1} = \Phi_n,\quad \forall t\in \real,\; \forall n\in\mathbb{N};
$$
\item For any $x\in\R$, $\Phi_n(x)$ converges to $x$ $\sigma$-strongly as $n\rightarrow\8$;
\item $\bigcup_{n\ge1} \mathcal{R}_n$ is $\sigma$-weakly dense in $\mathcal{R}$.
 \end{enumerate}

Since $\mathcal{D}$ is $\sigma^\varphi_t$-invariant and $\sigma^\varphi_t|_{\mathcal{D}}$
is exactly the modular automorphism
group of $\varphi|_{\mathcal{D}}$, we do not need to distinguish $\varphi$ and $\varphi|_{\mathcal{D}}$, $\sigma^\varphi_t$ and $\sigma^\varphi_t|_{\mathcal{D}}$, respectively. Now let
$$
\widehat{\mathcal{D}} = \mathcal{D} \rtimes_{\sigma^\varphi} G.
$$
Then $\widehat{\mathcal{D}}$ is naturally identified as
a von Neumann subalgebra of $\mathcal{R}$, generated by all operators $\pi(x),\; x \in \mathcal{D}$ and
$\lambda(t),\; t \in G$. We can extend $\E$ to a normal
faithful conditional expectation $\widehat{\mathcal{E}}$ from $\mathcal{R}$ onto $\widehat{\mathcal{D}}$, which is uniquely determined
by
\beq\label{conditional expectation}
\widehat{\mathcal{E}}(\lambda(t)x)=\lambda(t)\mathcal{E}(x),\quad\forall x\in\M,\;\forall t\in G.
\eeq
This conditional expectation satisfies
$$
\widehat{\mathcal{E}}\circ\Phi_n =\Phi_{n}\circ\widehat{\mathcal{E}},\quad \forall n\in\mathbb{N}.
$$
Now let  $n\in\mathbb{N}$ and $\mathcal{D}_n= \mathcal{R}_n \cap \widehat{\mathcal{D}}$. Note that $\Phi_n|_{\widehat{\mathcal{D}}}$
and $ \widehat{\mathcal{E}}|_{\mathcal{R}_n}$
are normal faithful conditional expectations from $\widehat{\mathcal{D}}$ onto
$\mathcal{D}_n$, respectively, from $\mathcal{R}_n$ onto $\mathcal{D}_n$.

Since $\A$ is $\sigma_{t}^{\varphi}$-invariant, by \eqref{representation1}, the family of all linear combinations on
$\lambda(t) \pi(x)$, $t \in G,\; x \in \A$, is a *-subalgebra of $\R$. Let $ \widehat{\mathcal{A}}$ be its $\sigma$-weakly closure  in
$\mathcal{R}$, i.e.,
\beq\label{defi:Ahat}
\widehat{\mathcal{A}}=\overline{span\{\lambda(t) \pi(x):\;t \in G,\; x \in \A\}}^{\sigma-weakly },
\eeq
 and let $\mathcal{A}_n =\widehat{\mathcal{A}}\cap \mathcal{R}_n$. By Lemma 3.1-3.3 in \cite{X}, $ \widehat{\mathcal{A}}$ (resp.  $\mathcal{A}_n$) is a maximal
subdiagonal algebra of $\mathcal{R}$ (resp. $\mathcal{R}_n$) with respect to $ \widehat{\mathcal{E}}$ (resp. $ \widehat{\mathcal{E}}|_{\mathcal{R}_n}$).

The next result is known. For easy reference, we give its proof ( see \cite{J1,X}).
 \begin{lemma}\label{sequence-sudgiagonal} Let $\M,\; \A,\;\E,\;\R,\;\widehat{\A},\;\widehat{\E},\;\Phi,\;\A_n,\;\Phi_n\; (\forall n\in\mathbb{N})$ be fixed as in the above. Then
\begin{enumerate}
\item $\Phi(\widehat{\A}\:)=\A,\quad\Phi_n(\widehat{\A}\:)=\A_n, \quad \forall n\in\mathbb{N}$;
\item $\bigcup_{n\ge1} \A_n$ is $\sigma$-weakly dense in $\widehat{\A}$;
\item $\widehat{\E}(\M)\subset\M$.

\end{enumerate}
 \end{lemma}
\begin{proof}
(1) By \eqref{conditional-expectation1}, we know  that $\Phi(\widehat{\A}\:)=\A$. Since $\A$ is $\sigma_{t}^{\varphi}$-invariant, by \eqref{eq:modular} and \eqref{eq:conditionexpectation-n},
we conclude that $\widehat{\A}$ is $\Phi_n$-invariant for all $n\in\mathbb{N}$. It is clear that
 $$
 \A_n=\Phi_n(\A_n)\subset\Phi_n(\widehat{\A}\:)\subset \widehat{\A} \cap\R_n=\A_n,
 $$
 so it follows that $\A_n=\Phi_n(\widehat{\A}\:)$.

(2) For any $x\in\widehat{\A}$, $\Phi_n(x)$ converges to $x$ $\sigma$-strongly as $n\rightarrow\8$. Hence, (2)  follows from (1).

(3) follows from \eqref{conditional expectation}.

\end{proof}

As  an extension of Theorem 3.1 in \cite{H2}, we have the following.

 \begin{theorem}\label{lem:denseHp} Let $\M,\; \A,\;\E,\;\R,\;\widehat{\A},\;\widehat{\E},\;\Phi,\;\A_n,\;\Phi_n\; (\forall n\in \mathbb{N})$ be fixed as in the above and let $0 < p < \8$.
\begin{enumerate}
\item $\{H^p(\mathcal{A}_n)\}_{n\ge1}$ is a  increasing sequence of subspaces of $H^p(\widehat{\A}\:)$;
\item $\bigcup_{n\ge1}H^p(\mathcal{A}_n)$  is dense in  $H^p(\widehat{\A}\:)$;
\item For each $n$, $H^p(\mathcal{A}_n)$ is isometric to the usual
noncommutative $H^p$-space associated with a finite
subdiagonal algebra.
\item $H^p(\mathcal{A})$  and all   $H^p(\mathcal{A}_n)$
are 1-complemented in $H^p(\widehat{\A}\:)$ for $1\le p <\8$.
\end{enumerate}
 \end{theorem}
\begin{proof}

(1) By Proposition \ref{pro:Hp-subspace}, $H^p(\A)$ and all  $H^p(\A_n)$ are naturally isometrically identified as subspaces of  $H^p(\widehat{\A}\:)$ for  $0 < p < \8$, and the sequence
$\{H^p(\mathcal{A}_n)\}_{n\ge1}$ is increasing.

(2) By Lemma
2.2 in \cite{Ju}, $\cup_{n\ge1}H^p(\mathcal{A}_n)$  is dense in  $H^p(\widehat{\A}\:)$  for $0 < p < \8$.

(3) Since $\mathcal{R}_n$ is a finite
von Neumann algebra, the result holds (see Remark \ref{rk:independent}).

(4) Using (1) of Lemma \ref{sequence-sudgiagonal} and Lemma 2.2 in \cite{JX} we obtain the desired results.
\end{proof}
Let
$$1\le p\le\8,\;\frac{1}{p}+\frac{1}{q}=1.
$$
If $x\in L^{p}(\mathcal{M})$ and $\tr(xa)=0, \;\forall a\in H^q(\A)$, we write  $x\perp J(H^q(\A))$.
Similarly, we define $x\perp J(H^q_0(\A))$.

\begin{proposition}\label{carecterizationh1} Let $\mathcal{A}$ be a  maximal subdiagonal algebra of
$\mathcal{M}$ with respect to $\mathcal{E}.$ Then, we have
\begin{enumerate}
\item
$
 \mathcal{A}=\{x\in\mathcal{M}:\;x\bot J(H^{1}_{0}(\A))\},\quad\mathcal{A}_{0}=\{x\in\mathcal{M}:\;x\bot J(H^{1}(\A))\}.
$
\item
$
H^{1}(\A)=\{x\in L^{1}(\mathcal{M}):\;x\bot J(\mathcal{A}_{0})\},\quad H^{1}_{0}(\A)=\{x\in L^{1}(\mathcal{M}):\;x\bot J( \mathcal{A})\}.
$
\end{enumerate}
\end{proposition}
\begin{proof} (i)  Assume $x\in \mathcal{A},$ if $y^{\ast}\in J(H^{1}_{0}(\A))$, then there
exist $\{a^{\ast}_{n}\}\subset \mathcal{A}_{0}$ such that
$y=\lim_{n\rightarrow\infty}a_{n}D$. Hence, by \eqref{eq:conditional expectation},
$$
tr(xy)=\lim_{n\rightarrow\infty}tr(xa_{n}D)=
\lim_{n\rightarrow\infty}tr(\mathcal{E}(xa_{n})D)=0,
$$
whence $\mathcal{A}\subset \{x\in\mathcal{M},\;x\bot J(H^{1}_{0}(\A))\}$.

Conversely, we take $x\in\mathcal{M},\;x\bot J(H^{1}_{0}(\A))$. By \eqref{equalityHp},
$$
tr(xD^\frac{1}{2}aD^\frac{1}{2})=0,\; \forall a\in  \mathcal{A}_{0}.
$$
 Hence
$$
tr(xD^{\frac{1}{2}}abD^{\frac{1}{2}})=0,\quad\forall b\in\A_0,\;\forall a\in \mathcal{A}.
$$
Using \eqref{equalityHp}, we obtain that $xH^{2}(\A)\bot J( H^{2}_{0}(\A))$. Since
$$
L^{2}(\mathcal{M})=H^{2}(\A)\oplus J(H^{2}_0(\A)),\quad xH^{2}(\A)\subseteq  H^{2}(\A).
$$
By Theorem 2.2 in \cite{JS},
$x\in\A$. Thus we proved the
first equality. The proof of second equality is similar to that of the first equality.

(ii)  It is clear that $H^{1}(\A)\subset\{x\in L^{1}(\mathcal{M}):\;x\bot J(\mathcal{A}_{0})\}$. Conversely, if $x\in L^{1}(\mathcal{M})$, $x\bot J(\mathcal{A}_{0})$ and $x\notin H^{1}(\A)$, then there
exist $y\in \mathcal{M}$ such that $tr(y^{\ast}x)=1$ and $y\bot
H^{1}(\A)$. By (i), $y^{\ast}\in
\mathcal{A}_{0}$. Therefore, $tr(y^{\ast}x)=0$. This
is a contradiction, so the first equality holds.   The proof of the second equality is similar.
\end{proof}

By Theorem 3.3 in \cite{J2}, we know that
\be
L^{p}(\mathcal{M}))=H^{p}(\A)\oplus J(H^{p}_0(\A))
\ee
for all $1< p<\infty$.
Hence, if  $\frac{1}{p}+\frac{1}{q}=1$, then
\beq\label{Hp-caracterization}
H^{p}(\A)=\{x\in L^{p}(\mathcal{M}):\;x\perp
J(H^{q}_{0}(\A))\}
\eeq
and
\be
  H^{p}_{0}(\A)=\{x\in L^{p}(\mathcal{M}):\;x\perp
J(H^{q}(\A))\}.
\ee
\begin{proposition}\label{charecterization-pro} Let  $1\le r,\;p,\;s\le\infty,\;\frac{1}{p}+\frac{1}{s}=\frac{1}{r}$ and $0\le\eta\le1$. Then
\begin{enumerate}
\item  $H^{p}(\A)=\{a\in L^{p}(\mathcal{M}):\;
D^{\frac{\eta}{s}}aD^{\frac{1-\eta}{s}}\in H^{r}(\A)\}.$
\item $H^{p}_{0}(\A)=\{a\in L^{p}(\mathcal{M}):\; D^{\frac{\eta}{s}}aD^{\frac{1-\eta}{s}}\in
H^{r}_{0}(\A)\}.$
\end{enumerate}
\end{proposition}
\begin{proof} We prove only (i). The proof of (ii) is similar. It is
clear that
$$
H^{p}(\A)\subseteq \{a\in L^{p}(\mathcal{M}):\;
D^{\frac{\eta}{s}}aD^{\frac{1-\eta}{s}}\in H^{r}(\A)\}.
$$
If $x\in
\{a\in L^{p}(\mathcal{M}):\; D^{\frac{\eta}{s}}aD^{\frac{1-\eta}{s}}\in
H^{r}(\A)\}$, then
$$
tr(D^{\frac{\eta}{s}}xD^{\frac{1-\eta}{s}}b)=0,\quad \forall
b\in H^{r'}_{0}(\A),
$$
where $\frac{1}{r}+\frac{1}{r'}=1$. Therefore, by \eqref{equalityHp},
$$
tr(xb)=0,\quad\forall b\in H^{q}_{0}(\A),
$$
where $\frac{1}{p}+\frac{1}{q}=1$.
Using \eqref{Hp-caracterization} or Lemma \ref{carecterizationh1} we get $x\in H^{p}(\A)$, hence
 the desired result holds.
\end{proof}

\section{ Complex interpolation theorem of the Haagerup
 $H^{p}$-spaces}

First we briefly recall the complex interpolation method will use in this
section. Our main reference is \cite{BL}. Let $X=(X_{0},X_{1})$ be a pair of
two compatible Banach spaces with norms $\|.\|_{X_{0}}=\|.\|_{0}$ and
$\|.\|_{X_{1}}=\|.\|_{1},$ respectively. The algebraic sum
$\Sigma(X_{0},X_{1})=X_{0}+X_{1}$ is a Banach space under the norm
$$
\|x\|_{\Sigma}=\inf\{\|x_{0}\|_{0}+\|x_{1}\|_{1}:\;x=x_{0}+x_{1},\;x_{j}\in
X_{j},\;j=0,1 \}.
$$
One defines a space  of certain $\Sigma(X_{0},X_{1})$-valued functions on the
strip $0\leq\mbox{Re}z\leq1$ as $F(X_{0},X_{1})$, which consists of functions
$f:0\leq\mbox{Re}z\leq1\mapsto\sum(X_{0},X_{1})$ satisfying:
\begin{enumerate}
  \item $f$ is bounded and continuous, and analytic in the interior( with respect to
$\|.\|_{\Sigma}$ ),
  \item $f(j+it)\in X_{j},\;t\in\mathbb{R},\;j=0,1,$
  \item for $j=0,1,$ the map $t\in\mathbb{R}\mapsto f(j+it)\in X_{j}$ is
$\|.\|_{j}$-continuous and $\lim_{t\rightarrow\pm\infty}\|f(j+it)\|_{j}=0.$
\end{enumerate}

It follows from the Phragm\'{e}n-Lindel\"{o}f theorem that the space $F(X_{0},X_{1})$
is Banach space under the norm
$$
\||f|\|=\max\{\sup_{t\in\mathbb{R}}\|f(it)\|_{0},\;\sup_{t\in\mathbb{R}}\|f(1+it)\|_{1}\}.
$$

\begin{definition}  For each $0<\theta<1,$ the complex interpolation
space $C_{\theta}(X_{0},X_{1})$ is the set of all $f(\theta),\;f\in
F(X_{0},X_{1}),$ equipped with the complex interpolation norm
$$
\|x\|_{C_{\theta}(X_{0},X_{1})}=\inf\{\||f|\|:\;f\in F(X_{0},X_{1}),\;f(\theta)=x\}.
$$
\end{definition}

 As in \cite{Kos},
for each $0\leq\eta\leq1$,   we define the imbedding $i_{\8}^{\eta}: \M\hookrightarrow L^{p_0}(\mathcal{M})$
by
$$
  i_{\8}^{\eta}(x)=D^{\eta}xD^{1-\eta},\quad \forall x\in \M.
$$
Let  $\M^\eta=i_{\8}^{\eta}(\M)=D^{\eta}\M D^{1-\eta}$. Then  $\M^\eta\subseteq L^{1}(\mathcal{M})$.
We  write (unless confusion occurs)
$$
 \|i_{\8}^{\eta}(x)\|_{\8}^{\eta}=
 \|D^{\eta}xD^{1-\eta}\|_{\8}^{\eta}=\|x\|_{\8},\quad\|x\|_{1}^\eta=\|D^{\eta}xD^{1-\eta}\|_{1},\quad\forall
 x\in \M.
 $$
 In what follows, we never consider a power of $\|\; \|_{\8},\;\|\; \|_{1}$, so that the above notations will
never make confusion. Since
$$
\|i_{\8}^{\eta}(x)\|_{1}^{\eta}=\|D^{\eta}xD^{1-\eta}\|_{1}\leq\|x\|_{\8}=
 \|D^{\eta}xD^{1-\eta}\|_{\8}^{\eta}=\|i_{\8}^{\eta}(x)\|_{\8}^{\eta}
$$
for all $x\in \M$,
$\|.\|_{1}^{\eta}\leq\|.\|_{\8}^{\eta}$ on the subspace $\M^\eta$
of $L^{1}(\mathcal{M})$.  Therefore, the pair $(\M^\eta, L^{1}(\mathcal{M}))$ is compatible and satisfies
$$
\Sigma(\M^\eta, L^{1}(\mathcal{M}))=L^{1}(\mathcal{M}),\quad
 \M^\eta\cap L^{1}(\mathcal{M})=\M^\eta, \quad
   \|.\|_{\Sigma}=\|.\|_{1}.
$$
If $1<p<\8$, then for each $0\leq\eta\leq1$,  we imbed $L^{p}(\mathcal{M})$ into
$L^{1}(\mathcal{M})$ via
$$
  i_{p}^{\eta}: x\in L^{p}(\mathcal{M})\rightarrow D^{\frac{\eta}{q}}xD^{\frac{1-\eta}{q}}\in
  L^{1}(\mathcal{M}),
$$
where $\frac{1}{p}+\frac{1}{q}=1$. We have that $i_{p}^{\eta}(
L^{p}(\mathcal{M}))=D^{\frac{\eta}{q}}L^{p}(\mathcal{M})D^{\frac{1-\eta}{q}}\subseteq L^{1}(\mathcal{M})$.
We also write (unless confusion occurs)
$$
 \|i_{p}^{\eta}(x)\|_{p}^{\eta}=
 \|D^{\frac{\eta}{q}}xD^{\frac{1-\eta}{q}}\|_{p}^{\eta}=\|x\|_{p},\quad\forall
 x\in L^{p}(\mathcal{M}).
 $$
 By Theorem 9.1 in \cite{Kos},  we have
\beq\label{interpolation}
C_{\theta}(\M^\eta, L^{1}(\mathcal{M}))=i_{p}^{\eta}( L^{p}(\mathcal{M}))
\eeq
with equal norms, where
$\theta=\frac{1}{p}$.

As in the above, for each $0\leq\eta\leq1$ , we can imbed $\A$ into
$H^{1}(\mathcal{A})$ via
$$
  i_{\8}^{\eta}: a\in \A\rightarrow D^{\eta}aD^{1-\eta}\in
  H^{1}(\mathcal{A}).
$$
Let $\A^\eta=i_{\8}^{\eta}(\A)=D^{\eta}\A D^{1-\eta}$. We write
 $$
  \|i_{\8}^{\eta}(a)\|_{\8}^{\eta}=
 \|D^{\eta}aD^{1-\eta}\|_{\8}^{\eta}=\|a\|_{\8},\quad \|a\|_{1}^\eta=\|D^{\eta}aD^{1-\eta}\|_{1},\quad \forall a\in \A.
 $$
  Hence
$(\A^\eta, H^{1}(\mathcal{A}))$ is compatible.

For the Haagerup
noncommutative $H^{p}$-space, we obtain the following  result.
\begin{lemma} \label{interpolation-maximalhp} Let $\R$ and $\widehat{\A}$ be as in \eqref{defi:R} and \eqref{defi:Ahat}, respectively.  Then for $1<p<\infty$,
$$
C_{\theta}(\widehat{\A}\:^\eta,
H^{1}(\widehat{\A}\:))=i_{p}^{\eta}( H^{p}(\widehat{\A}\:))
$$
with equal norms, where $\theta=\frac{1}{p}$.
\end{lemma}

\begin{proof}
By \eqref{interpolation},  $C_{\theta}(\R^{\eta},
L^{1}(\R))=i_{p}^{\eta}( L^{p}(\R))$
with equal norms.   We have that
$$
C_{\theta}(\widehat{\A}\:^\eta,
H^{1}(\widehat{\A}\:))\subset C_{\theta}(\R^{\eta},
L^{1}(\R))=i_{p}^{\eta}( L^{p}(\R)).
$$
 On the other hand,  $C_{\theta}(\widehat{\A}\:^\eta,
H^{1}(\widehat{\A}\:))\subset H^{1}(\widehat{\A}\:)$. By Proposition \ref{charecterization-pro},
 $$
 C_{\theta}(\widehat{\A}\:^\eta,
H^{1}(\widehat{\A}\:))\subset i_{p}^{\eta}( H^{p}(\widehat{\A}\:));
 $$
whence
\beq\label{inequality1}
\|x\|_p\le\|x\|_{ C_{\theta}(\widehat{\A}\:^\eta,\:
H^{1}(\widehat{\A}\:))},\quad \forall x\in  C_{\theta}(\widehat{\A}\:^\eta,
H^{1}(\widehat{\A}\:)).
\eeq
Let $\{H^p(\mathcal{A}_n)\}_{n\ge1}$ be the sequence in Theorem \ref{lem:denseHp}. Since  $H^p(\mathcal{A}_n)$ is isometric to the usual
noncommutative $H^p$-space associated with a finite
subdiagonal algebra,
 $$
 C_{\theta}(\A_n,
H^{1}(\A_n))=H^{p}(\A_n)
 $$
with equal norms, for all $ n\in \mathbb{N}$ (see the remark following Lemma 8.5 in \cite{PX} or Theorem 3.3 in \cite{BO}).
Hence, by (1) of Lemma \ref{sequence-sudgiagonal} and (4) of Theorem \ref{lem:denseHp},
 $$
 H^{p}(\A_n)=C_{\theta}(\A_n,
H^{1}(\A_n))\subset C_{\theta}(\widehat{\A}\:^\eta, H^{1}(\widehat{\A}\:))
 $$
 and
 \beq\label{inequality2}
  \|a\|_{ C_{\theta}(\widehat{\A}\:^\eta,\: H^{1}(\widehat{\A}\;))}\le\|a\|_p,\quad \forall a\in H^{p}(\A_n).
 \eeq
 Let $x\in H^{p}(\widehat{\A}\:)$. Using (2) of Theorem \ref{lem:denseHp} we obtain   a sequence $\{a_n\}_{n\ge1}$ in $\bigcup_{n\ge1}H^{p}(\A_n)$ such that
 $$
\lim_{n\rightarrow\8}\| a_n -x\|_p=0.
 $$
Since $\{H^p(\mathcal{A}_n)\}_{n\ge1}$ is increasing, using \eqref{inequality2} we obtain that
$$
\|a_n-a_m\|_{ C_{\theta}(\widehat{\A}\:^{\eta},\: H^{1}(\widehat{\A}\:))}\le\|a_n-a_m\|_p,\quad n,m\ge1.
$$
Hence $\{a_n\}_{n\ge1}$ is a Cauchy sequence in $C_{\theta}(\widehat{\A}\:^{\eta}, H^{1}(\widehat{\A}\:))$.
 So, there exists  $y$ in $C_{\theta}(\widehat{\A}\:^{\eta}, H^{1}(\widehat{\A}\:))$ such that
$$
\lim_{n\rightarrow\8}\|a_n-y\|_{ C_{\theta}(\widehat{\A}\:^{\eta},\:H^{1}(\widehat{\A}\:))}=0
$$
By \eqref{inequality1}, $\lim_{n\rightarrow\8}\|a_n-y\|_{p}=0$. Therefore, $y=x$. We  deduce that
$$
  \|x\|_{ C_{\theta}(\widehat{\A}\:^{\eta},\:H^{1}(\widehat{\A}\:))}\le\|x\|_p,\quad \forall x\in H^{p}(\widehat{\A}\:).
$$
Thus $ C_{\theta}(\widehat{\A}\:^{\eta}, H^{1}(\widehat{\A}\:))= i_{p}^{\eta}(H^{p}(\widehat{\A}\:))$.
\end{proof}

\begin{theorem} \label{interpolationhp} Let $1<p<\infty$.
Then
$$
C_{\theta}(\A^{\eta},
H^1(\mathcal{A}))=i_{p}^{\eta}( H^{p}(\mathcal{A}))
$$
with equivalent norms, where $\theta=\frac{1}{p}$.
\end{theorem}
\begin{proof}
Using \eqref{interpolation} we obtain that $C_{\theta}(\M^{\eta},
L^{1}(\M))=i_{p}^{\eta}( L^{p}(\M))$
with equal norms.   It is clear that
$$
C_{\theta}(\A^\eta,
H^{1}(\A\:))\subset C_{\theta}(\M^{\eta},
L^{1}(\M))=i_{p}^{\eta}( L^{p}(\M)).
$$
Since  $C_{\theta}(\A^\eta,
H^{1}(\A))\subset H^{1}(\A)$, by Proposition \ref{charecterization-pro}, we have that
$$
 C_{\theta}(\A,
H^{1}(\A))\subset i_{p}^{\eta}( H^{p}(\A)).
 $$
On the other hand, by (i) of Lemma \ref{sequence-sudgiagonal} and (iv) of Theorem \ref{lem:denseHp}, we have that
$\A^{\eta},\,H^1(\A),\;H^{p}(\mathcal{A})$ are 1-completed in $\widehat{\A}\:^\eta,\;
H^{1}(\widehat{\A}\:),\; H^{p}(\widehat{\A}\:)$, respectively. Hence,
$$
\|x\|_{ C_{\theta}(\A,\:
H^{1}(\A))}= \|x\|_{ C_{\theta}(\widehat{\A}\:^\eta,\: H^{1}(\widehat{\A}\;))},\quad \forall x\in  C_{\theta}(\A^\eta,
H^{1}(\A))
$$
and
$$
\|x\|_{ H^{p}(\widehat{\A}\;)}=\|x\|_{H^p(\A)},\quad \forall x\in
H^{p}(\A).
$$
 Using Lemma \ref{interpolation-maximalhp} we get
\beq\label{equality-1}
\|x\|_{ C_{\theta}(\A,\:
H^{1}(\A))}=\|x\|_p,\quad \forall x\in  C_{\theta}(\A^\eta,
H^{1}(\A)).
\eeq

Let $a\in\A$. For $\delta>0$, we define
$$
f_{a}(z)=e^{\delta(z^2-\theta^2)}D^{\eta(1-z)}aD^{z}D^{(1-\eta)(1-z)},
\; 0\leq\mbox{Re}z\leq1.
$$
Obviously,
$$
 f_{a}(it)\in \mathcal{A}^{\eta},\quad\| f_{a}(it)\|_{\infty}^{\eta}\le\|a\|_{\8},
 $$
 $$
f_{a}(1+it)\in H^1(\A),\quad \| f_{a}(1+it)\|_{1}\leq e^{\delta(1-\theta^2)}\|aD\|_{1}\le e^{\delta(1-\theta^2)}\|a\|_{\8},
 $$
  $$
 f_{a}(z)\in
 F(\mathcal{A}^{\eta},H^{1}(\A)),\quad
 \||f_{a}|\|\le e^{\delta(1-\theta^2)}\|a\|_{\8},\quad f_{a}(\theta)=D^{\frac{\eta}{q}}aD^\frac{1}{p}D^{\frac{1-\eta}{q}}.
$$
Thus $i_{p}^{\eta}(aD^\frac{1}{p})\in  C_{\theta}(\A^\eta,H^{1}(\A))$.
Therefore, applying \eqref{equality-1} we deduce that
\beq\label{equality-2}
\|aD^{\frac{1}{p}}\|_{ C_{\theta}(\A,\:
H^{1}(\A))}=\|aD^{\frac{1}{p}}\|_p,\quad \forall a\in  \A.
\eeq
$
C_{\theta}(\A^{\eta},H^1(\mathcal{A}))=i_{p}^{\eta}( H^{p}(\mathcal{A}))$ is proved similarly as  Lemma \ref{interpolation-maximalhp} by using  \eqref{equality-2} and Definition \ref{def:hp}. We omit the details.
\end{proof}

Let
$H^{p}(\mathcal{A})_{L}=C_{\theta}(\mathcal{A}^{0},H^{1}(\mathcal{A}))$ and $H^{p}(\mathcal{A})_{R}=C_{\theta}(\mathcal{A}^{1},H^{1}(\mathcal{A}))$.
We call $H^{p}(\mathcal{A})_{L}$  and $H^{p}(\mathcal{A})_{R}$ the left and the right  $H^{p}$-spaces with respect to $D$, respectively. By Theorem \ref{interpolationhp}, we have
that
$$
H^{p}(\mathcal{A})_{L}=H^{p}(\mathcal{A})D^{\frac{1}{q}},\quad H^{p}(\mathcal{A})_{R}=D^{\frac{1}{q}}H^{p}(\mathcal{A}).
$$

We have the following Stein-Weiss type
interpolation theorem.

\begin{theorem}  Let $1< p<\infty,\;\frac{1}{p}+\frac{1}{q}=1,\;0\le\eta\le1$. Then
$$
C_{\eta}(H^{p}(\mathcal{A})D^{\frac{1}{q}},D^{\frac{1}{q}}H^{p}(\mathcal{A}))=
D^{\frac{\eta}{q}}H^{p}(\mathcal{A})D^{\frac{1-\eta}{q}}.
$$
Consequently,
$$
C_{\eta}(H^{p}(\mathcal{A})_{L},H^{p}(\mathcal{A})_{R})=
C_{\theta}(\mathcal{A}^{\eta},H^{1}(\mathcal{A}))\mid_{\theta=\frac{1}{p}}
$$
\end{theorem}
\begin{proof}
By Theorem 11.2 of \cite{Kos}, we see that
$$
C_{\eta}(L^{p}(\M)D^{\frac{1}{q}},D^{\frac{1}{q}}L^{p}(\M))=
D^{\frac{\eta}{q}}L^{p}(\M)D^{\frac{1-\eta}{q}}.
$$
Thus
$$
C_{\eta}(H^{p}(\A)D^{\frac{1}{q}},D^{\frac{1}{q}}H^{p}(\A))\subset
D^{\frac{\eta}{q}}L^{p}(\M)D^{\frac{1-\eta}{q}}.
$$
It is trivial that $C_{\eta}(H^{p}(\A)D^{\frac{1}{q}},D^{\frac{1}{q}}H^{p}(\A))\subset H^1(\A)$.  By Proposition \ref{charecterization-pro}, we deduce
\beq\label{inclusion-1}
C_{\eta}(H^{p}(\A)D^{\frac{1}{q}},D^{\frac{1}{q}}H^{p}(\A))\subset
D^{\frac{\eta}{q}}H^{p}(\A)D^{\frac{1-\eta}{q}}.
\eeq

Conversely, assume $x\in H^{p}(\A)$.  Let $\delta>0$. Define for $0\leq\mbox{Re}z\leq1$
$$
f(z)=e^{\delta(z^2-\eta^2)}D^{\frac{z}{q}}xD^{\frac{1-z}{q}}.
$$
It is clear that
$$
 f(it)\in H^{p}(\mathcal{A})D^{\frac{1}{q}},\quad\| f(it)\|_{p}^{0}\le\|x\|_{p},
 $$
 $$
 f(1+it)\in D^{\frac{1}{q}}H^{p}(\mathcal{A}),\quad \| f(1+it)\|_{p}^1\leq e^{\delta(1-\eta^2)}\|x\|_{p},
 $$
  $$
 f(z)\in
 F(H^{p}(\mathcal{A})D^{\frac{1}{q}},D^{\frac{1}{q}}H^{p}(\mathcal{A})),\quad
 \||f|\|\le e^{\delta(1-\eta^2)}\|x\|_{p},\quad f(\eta)=D^{\frac{\eta}{q}}xD^{\frac{1-\eta}{q}}.
$$
Thus $i_{p}^{\eta}(x)\in  C_{\eta}(H^{p}(\mathcal{A})D^{\frac{1}{q}},D^{\frac{1}{q}}H^{p}(\mathcal{A}))$.
It follows that
\beq\label{inclusion-2}
D^{\frac{\eta}{q}}H^{p}(\A)D^{\frac{1-\eta}{q}}\subset C_{\eta}(H^{p}(\A)D^{\frac{1}{q}},D^{\frac{1}{q}}H^{p}(\A)).
\eeq
Therefore, by \eqref{inclusion-1} and \eqref{inclusion-2}, the desired result holds. The second result follows immediately by virtue of Theorem \ref{interpolationhp}.
\end{proof}

If $1< p<\infty,\;\frac{1}{p}+\frac{1}{q}=1$, the duality between $i_{p}^{\eta}( H^{p}(\mathcal{A}))$
and $i_{q}^{\eta}( J(H^{q}(\mathcal{A})))$ is defined by
$(i_{p}^{\eta}(a),i_{q}^{\eta}(b))\mapsto tr(ab)$, for which
$$
(i_{p}^{\eta}( H^{p}(\mathcal{A})))'=i_{q}^{\eta}(J( H^{q}(\mathcal{A})))
$$
isometrically ( see Corollary 3.4 in \cite{J2}).
Let $BMO(\A)$ be the noncommutative BMO  space defined in \cite{J2}. We have that $H^1(\A)^*=BMO(\A)$ (see Theorem 3.10 in \cite{J2}).
\begin{theorem}  Let $0\le\eta\le1 $. Then
$$
C_{\frac{1}{p}}(BMO(\mathcal{A}),H^{1}(\mathcal{A})) =i_{p}^{\eta}(
H^{p}(\mathcal{A})),\quad \forall p\in(1,\infty).
$$
\end{theorem}

\begin{proof} By Theorem \ref{interpolationhp} and the reiteration theorem (\cite{BL}, Theorem
4.6.1), for $1\leq p_{0},\;p_{1},\;p<\infty$ and $0<\theta<1$ with
$\frac{1}{p}=\frac{1-\theta}{p_{0}}+\frac{\theta}{p_{1}}$ we have
\beq\label{interpolation-hp-hq}
C_{\theta}(i_{p_{0}}^{\eta}( H^{p_{0}}(\mathcal{A})), i_{p_{1}}^{\eta}(
H^{p_{1}}(\mathcal{A})))=i_{p}^{\eta}( H^{p}(\mathcal{A}))
\eeq
with equal norms.

Let $1< r< p<\infty$ and $\frac{1}{p}+\frac{1}{q}=1,\;\frac{1}{r}+\frac{1}{s}=1$.
Then  $1< q< s<\infty$. By \eqref{interpolation-hp-hq}, the following holds with equivalent norms
$$
C_{\frac{r}{p}}(i_{1}^{\eta}( H^{1}(\mathcal{A})), i_{s}^{\eta}(
H^{s}(\mathcal{A})))=i_{q}^{\eta}( H^{q}(\mathcal{A})).
$$
Applying Corollary 4.5.2 in \cite{BL} we obtain that
$$\begin{array}{rl}
 i_{p}^{\eta}( H^{p}(\mathcal{A})) & =(i_{q}^{\eta}(J(
H^{q}(\mathcal{A}))))'\\
&=(J(C_{\frac{r}{p}}(i_{1}^{\eta}( H^{1}(\mathcal{A})), i_{s}^{\eta}(
H^{s}(\mathcal{A})))))'\\
&=C_{\frac{r}{p}}(BMO(\mathcal{A}), i_{r}^{\eta}( H^{r}(\mathcal{A}))).
\end{array}
$$

Moreover we deduce from \eqref{interpolation-hp-hq} that
$$
i_{q}^{\eta}( H^{r}(\mathcal{A}))=C_{\frac{q}{s}}(i_{1}^{\eta}( H^{1}(\mathcal{A})), i_{p}^{\eta}(
H^{p}(\mathcal{A})))=C_{1-\frac{q}{s}}(i_{p}^{\eta}(H^{p}(\mathcal{A})),i_{1}^{\eta}( H^{1}(\mathcal{A}))).
$$

Using Wolff's interpolation theorem (see \cite{W})
 we obtain that
$$
C_{\frac{1}{p}}(BMO(\mathcal{A}),H^{1}(\mathcal{A})) =i_{p}^{\eta}(
H^{p}(\mathcal{A})).
$$

\end{proof}

\subsection*{Acknowledgement}  T.N. Bekjan is partially supported by NSFC grant No.11771372, M. Raikhan is partially supported by  project AP05131557 of the Science Committee of Ministry of Education and Science of the Republic of Kazakhstan.

\end{document}